\documentclass[12pt, a4paper]{amsart}

\usepackage{latexsym}

\usepackage{color}

\usepackage{amsthm,xparse}

\usepackage{amssymb}
\usepackage[centertags]{amsmath}
\usepackage{verbatim}

\usepackage[a4paper, hmargin=2.5cm, vmargin={3.5cm, 3.5cm}]{geometry}

\usepackage{units}

\usepackage[english]{babel}

\usepackage[colorlinks=true,linkcolor=red ,citecolor=blue]{hyperref}

\usepackage{amsthm}

\newtheorem{theorem}{Theorem}

\NewDocumentEnvironment{manual}{O{theorem}m}
{%
	\addtocounter{theorem}{-1}%
	\renewcommand\thetheorem{#2}%
	\begin{#1}
	}
	{\end{#1}}

\newtheorem{lemma}[equation]{Lemma}

\newtheorem{proposition}[equation]{Proposition}
\newtheorem{corollary}[equation]{Corollary}

\newtheorem{claim}[equation]{Claim}

\theoremstyle{definition}

\newcommand{\theoremname}{testing}
\theoremstyle{remark}
\newtheorem*{remark*}{Remark}

\numberwithin{equation}{section}

\newcommand{\e}{\varepsilon}

\newcommand{\eps}{\varepsilon}

\newcommand{\E}{\mathbb{E}}

\newcommand{\var}{\text{Var}}

\newcommand*{\dd}[1]{{\rm d}#1}

\newcommand\inner[2]{\langle #1, #2 \rangle}

\newcommand{\RomanNumeralCaps}[1]
{\rm{\MakeUppercase{\romannumeral #1}}}

\def\eqdef{\overset{\rm{def}}{=}}

\renewcommand{\ge}{\geqslant}

\newcommand{\bR}{\mathbb R}

\newcommand{\bZ}{\mathbb Z}

\newcommand{\bN}{\mathbb N}

\newcommand{\bB}{\mathbb B}

\newcommand{\ti}{\widetilde}

\title[Fluctuations in the Gaussian perturbations of the lattice]{Fluctuations of linear statistics for Gaussian perturbations of the lattice $\bZ^d$}
\author{Oren Yakir}
\thanks{Supported by ISF Grants 382/15, 1903/18 and by ERC Advanced Grant 692616}
\date{July 22, 2020. Revised February 4, 2021.}
\address{\tiny{School of Mathematics, Tel Aviv University,
		Tel Aviv 6997801, Israel.}}
\email{oren.yakir@gmail.com}
\begin{document}
	
	\maketitle
	
	\begin{abstract}
		We study the point process $W$ in $\bR^d$ obtained by adding an independent Gaussian vector to each point in $\bZ^d$. Our main concern is the asymptotic size of fluctuations of the linear statistics in the large volume limit, defined as
		\[
		N(h,R) = \sum_{w\in W} h\left(\frac{w}{R}\right),
		\]
		where $h\in \left(L^1\cap L^2\right)(\mathbb{R}^d)$ is a test function and $R\to \infty$. We will also consider the stationary counter-part of the process $W$, obtained by adding to all perturbations a random vector which is uniformly distributed on $[0,1]^d$ and is independent of all the Gaussians.  We focus on two main examples of interest, when the test function $h$ is either smooth or is an indicator function of a convex set with a smooth boundary whose curvature does not vanish.
	\end{abstract}
		
	\section{Introduction and the main results} 
	We consider the random point process in $\bR^d$
	\begin{equation*}
	\label{eq:definition_of_the_process}
	W = \left\{n + \xi_{n} \mid n\in \mathbb{Z}^d\right\},
	\end{equation*}
	where $\left\{\xi_n\right\}_{n\in \mathbb{Z}^d}$ are independent and identically distributed symmetric Gaussian random vectors with density
	\begin{equation}
	\label{eq:d_gaussian_measure}
	\phi_a(x) = (a\pi)^{-d/2}e^{-|x|^2/a}, \qquad a>0
	\end{equation} 
	with respect to the Lebesgue measure on $\bR^d$. We will also be interested in the \emph{stationarized} version of the process $W$, defined as
	\begin{equation*}
	\label{eq:definition_of_the_stat_process}
	W_{\sf s} = \left\{n + \xi_{n} + \zeta \mid n\in \mathbb{Z}^d\right\}
	\end{equation*}
	where the sequence of random vectors $\{\xi_{n}\}_n$ is the same as before and $\zeta$ is a random vector uniformly distributed on $[0,1]^d$ and independent of all the $\xi_{n}$. We note that the distribution of the random set $W_{\sf s}$ is invariant with respect to all translations of $\bR^d$ (also known as stationary), while the distribution of the random set $W$ is only invariant with respect to translations by points of $\bZ^d$. 
	
	We represent the point process $W$ as a random measure, given by 
	\begin{equation}
	\label{eq:definition_of_random_measure}
	\mathbf{n} = \sum_{w\in W} \delta_w,
	\end{equation}
	where $\delta_x$ is a unit point mass at the point $x\in \mathbb{R}^d$. Then, a common way of studying the asymptotic behavior of $W$ is to introduce the random variable
	\begin{equation}
	\label{eq:def_linear_stat}		
	N(h,R) \eqdef \int_{\bR^d} h\left(\frac{x}{R}\right) \, {\rm d}\mathbf{n}(x) = \sum_{w\in W} h\left(\frac{w}{R}\right)
	\end{equation}
	called the \emph{linear statistics} of $W$. Here, $h\in\left(L^1\cap L^2 \right)\left(\mathbb{R}^d\right)$ is a test function and $R>0$ is a large parameter. For the stationary process $W_{\sf s}$, we denote by $\mathbf{n}_{\sf s}$ and $N_{\sf s}(h,R)$ the induced measure (\ref{eq:definition_of_random_measure}) and the linear statistics (\ref{eq:def_linear_stat}) defined in a similar way (with the sum in both (\ref{eq:definition_of_random_measure}) and (\ref{eq:def_linear_stat}) running on $W_{{\sf s}}$ instead of $W$).
	\begin{remark*}
		Although the values of the linear functional $h\mapsto N(h,R)$ depends on the choice of the representative $h\in L^1(\bR^d)$, the distribution of $N(h,R)$ (as a random variable) do not change after redefining $h$ on a set of measure zero. As we will only be interested in the statistical properties of $N(h,R)$, we will neglect this issue throughout the paper.
	\end{remark*}
	\subsection{The mean}
	Denote by $\E(X)$ the expectation of a random variable $X$. It is not surprising (see Corollary \ref{cor:mean_linear_statistics_first_order}) that for all test functions $h\in \left(L^1\cap L^2\right)\left(\mathbb{R}^d\right)$,
	\begin{equation}
	\label{eq:mean_linear_statistic_asymptotics}
	\E\left[N(h,R)\right]  = \left(\int_{\bR^d} h \, {\rm d}m_d + o(1)\right) R^{d}
	\end{equation} 
	as $R\to\infty$. Here and throughout $m_d$ is the Lebesgue measure on $\bR^d$. We also note that, since $W_{\sf s}$ has a translation-invariant distribution with unit intensity,
	\begin{equation}
		\label{eq:mean_linear_statistics_stat}
		\E\left[N_{\sf s}(h,R)\right] = R^d \int_{\bR^d} h \, {\rm d}m_d.
	\end{equation}
	Denote by $K_R \eqdef \left\{Rx \mid x\in K\right\}$ the dilation of a bounded domain $K\subset \mathbb{R}^d$. By definition, we have
	$
	\mathbf{n}\left(K_R\right) = N\left(\mathbf{1}_K,R\right)
	$
	where $\mathbf{1}_K$ is the indicator function of $K$, so (\ref{eq:mean_linear_statistic_asymptotics}) yields that
	\begin{equation}
		\label{eq:limit_mean_indicator_of_domains}
		\lim_{R\to\infty} \frac{\E\left[\mathbf{n}\left(K_R\right)\right]}{R^d} = m_d(K).
	\end{equation} 
	We will be interested in the remainder term in (\ref{eq:limit_mean_indicator_of_domains}). In Section \ref{sec:mean_for_linear_statistics} we prove that
	\[
	\E\left[\mathbf{n}\left(K_R\right)\right] = m_d(K) R^d + \mathcal{O}\left(R^{(d-1)/2}\right)
	\]
	provided that $K$ is a compact convex set such that $\partial K$ has nowhere vanishing Gaussian curvature (for the definition of Gaussian curvature of a surface, see for example \cite[Chapter~1.2, p.~49-50]{sogge}). The assumption on the Gaussian curvature is essential; we will show in Section \ref{sec:mean_number_of_points_in_large_cube} that if $K=[-\frac{1}{2},\frac{1}{2}]^d$ is the unit cube then the remainder term in (\ref{eq:limit_mean_indicator_of_domains}) can be as large as $R^{d-1}$.
	
	In view of (\ref{eq:mean_linear_statistics_stat}), it may seem that the additional uniform perturbation introduced in $W_{\sf s}$ ``regularizes" the mean and suppresses fluctuations. In what follows we will show that this is \emph{not} the case, as the variance of linear statistics of $W_{\sf s}$ can be much larger than the same variance with respect to $W$.
	
	\subsection{Fluctuations of linear statistics} We will be interested in determining the asymptotic of $\var\left(N\left(h,R\right)\right)$ as $R$ tends to infinity for various classes of test functions $h\in \left(L^1\cap L^2\right)\left(\mathbb{R}^d\right)$. Recall that $\var\left(X\right)$ is the variance of a random variable $X$, defined as
	$$
	\var\left(X\right) = \E\left[\left(X - \E\left[X\right]\right)^2\right].
	$$ 
	Unlike relation (\ref{eq:mean_linear_statistic_asymptotics}) for the  mean, the leading order asymptotic of the variance depends on smoothness properties of the test function $h$. We will prove that for $K\subset \bR^d$ we have
	\[
	\text{Var}(\mathbf{n}(K_R)) \sim c_a R^{d-1} \times (\text{Surface area of } \partial K)
	\]
	provided that $K$ is a compact convex set such that $\partial K$ has nowhere vanishing Gaussian curvature (see Theorem \ref{thm:variance_for_convex_sets}). In contrary to the point count, for smooth test functions $h$ such that $|\nabla h|\in L^2(\bR^d)$ we show that
	\[
	\text{Var}(N(h,R)) \sim c_a R^{d-2}\times \int_{\bR^d} |\nabla h|^2
	\]
	see Theorem \ref{thm:variance_for_sobolev_space} below. 

	The starting point for both results mentioned above is an exact formula for the variance of general linear statistics (see Theorem \ref{thm:variance_linear_statstic}). Another application of Theorem \ref{thm:variance_linear_statstic} is an upper bound for $\text{Var}(N(h,R))$ which is valid for all test functions $h\in(L^1\cap L^2)(\bR^d)$, and interpolates between the case of smooth linear statistics and the case of point count (cf. Theorem \ref{thm:upper_bound_for_variance}).
	
	We normalize the Fourier transform of a function $f\in L^1\left(\mathbb{R}^d\right)$ as
	\[
	\widehat{f}(\lambda) = \int_{\mathbb{R}^d} f(x) e^{-2\pi i\inner{x}{\lambda}} {\rm{d}}m_{d}(x),
	\]
	where $\inner{\cdot}{\cdot}$ is the inner product in $\mathbb{R}^d$. The following result is a simple application of the Poisson summation formula (for the proof, see Section \ref{sec:fluctuations_of_linear_statistics}).
	\begin{theorem}
		\label{thm:variance_linear_statstic}
		For any $h\in\left(L^1\cap L^2\right)\left(\mathbb{R}^d\right)$, we have
		\begin{align*}
		{\normalfont \text{Var}} &\left(N(h,R)\right) \\ &= R^d\sum_{m\in \bZ^d} \left\{\int_{\mathbb{R}^d}\widehat{h}(\lambda)\widehat{h}(Rm-\lambda)\left(e^{-a\pi^2|m|^2} - e^{-a\pi^2|\lambda|^2/R^2}e^{-a\pi^2|Rm-\lambda|^2/R^2}\right)\dd m_{d}(\lambda)\right\}. 
		\end{align*}
	\end{theorem}
	\begin{remark*}
		The Gaussian nature of the perturbations is not essential for Theorem \ref{thm:variance_linear_statstic} to hold. For more general perturbations (e.g., the ones with a density) one can replace in Theorem \ref{thm:variance_linear_statstic} the terms $e^{-a\pi^2|\cdot|^2}$ with the corresponding characteristic function (the Fourier transform of the distribution of the perturbation).
	\end{remark*}
	Although not difficult, Theorem \ref{thm:variance_linear_statstic} is the starting point to all the results we present from now on. In fact, Theorem $\ref{thm:variance_linear_statstic}$ can be thought of as the ``Fourier side" of the equality
	\[
	\var\left(N(h,1)\right) = \iint_{\bR^d\times\bR^d} h(x)h(y) G(x,y) \, \dd m_d(x) \dd m_d(y)
	\]
	where $G:\bR^d\times\bR^d\to \bR$ is the \emph{(full) two-point function} as defined in \cite[Section~4]{ghosh_lebowitz_survey}. When considering the Gaussian perturbations of the lattice $W$, the corresponding two-point function is given by 
	\begin{equation}
		\label{eq:two_point_function}
		G(x,y) = \left(\sum_{n\in \bZ^d} \phi_a(x-n)\right) \delta(x-y) - \left(\sum_{n\in\bZ^d}\phi_a(x-n)\phi_a(y-n)\right)
	\end{equation}
	where $\phi_a$ is given by (\ref{eq:d_gaussian_measure}). For a proof of \eqref{eq:two_point_function} see also \cite[Appendix~C]{klatt_kim_torquato}.
	
	As a first application of Theorem \ref{thm:variance_linear_statstic}, we prove in Section $\ref{sec:fluctuations_of_linear_statistics}$ the following upper bound, valid for a large family of test functions. We write $A\lesssim B$ if there exist a positive constant $C$ so that $A\leq C\cdot B$. The constant $C$ may depend on the dimension $d$ and the dispersion parameter $a$.
	\begin{theorem}
		\label{thm:upper_bound_for_variance}
		Let $h\in \left(L^1\cap L^2\right)(\mathbb{R}^d)$. Then, as $R\to\infty$,
		\begin{equation}
		\label{eq:upper_bound_for_variance_general_statistic}
		{\normalfont \text{Var}}\left(N(h,R)\right) \lesssim R^{d-2} \int_{|\lambda| \leq R} |\widehat{h}(\lambda)|^2|\lambda|^2 {\rm d} m_d(\lambda) + R^d\int_{|\lambda| \ge R} |\widehat{h}(\lambda)|^2 {\rm d} m_d(\lambda).
		\end{equation}
	\end{theorem}
	Observe that the right hand side of (\ref{eq:upper_bound_for_variance_general_statistic}) implicitly interpolates between the $L^2$-norm of the function $h$ and the $L^2$-norm of $\nabla h$. As we will soon show (see Theorems \ref{thm:variance_for_sobolev_space} and \ref{thm:variance_for_convex_sets}), the first term in the sum dominates in the case of smooth test functions and the second term dominates when the test function is an indicator of a bounded domain. 
	
	Although we will prove Theorem \ref{thm:upper_bound_for_variance} directly, it is worth mentioning that one may obtain the upper bound (\ref{eq:upper_bound_for_variance_general_statistic}) by estimating directly the two-point function $G$ given in (\ref{eq:two_point_function}) (cf. \cite[Lemma~A.2]{sedimentation}). We also mention that a similar type of interpolation formula for the fluctuations of stationary zeros of the Gaussian analytic function appeared in the work of Nazarov and Sodin \cite[Theorem~1.1]{sodin_nazarov}. We do not know whether the corresponding lower bound to Theorem \ref{thm:upper_bound_for_variance} holds for arbitrary test functions $h$. 
	
	Moving on to consider the stationary version $W_{\sf s}$ of our process, we obtain the following formula which is another consequence of Theorem \ref{thm:variance_linear_statstic}.
	\begin{theorem}
		\label{thm:variance_linear_statstic_stat}
		For any $h\in\left(L^1\cap L^2\right)\left(\mathbb{R}^d\right)$ we have
		\begin{align*}
		\label{eq:variance_linear_statistic_stat}
		{\normalfont \text{Var}} &\left(N_{\sf s}(h,R)\right) \\ & = R^d\int_{\bR^d}|\widehat{h}(\lambda)|^2 \left(1-e^{-2a\pi^2|\lambda|^2/R^2} \right) \dd m_d(\lambda) + R^{2d}\sum_{m\in \bZ^d \setminus\{0\}} e^{-2a\pi^2|m|^2} |\widehat{h}(Rm)|^2.
		\end{align*}
	\end{theorem}
	Theorem $\ref{thm:variance_linear_statstic_stat}$ yields the lower bound
	\begin{equation*}
	\label{eq:lower_bound_for_variance_general_statistic_stat}
	{\normalfont \text{Var}}\left(N_{\sf s}(h,R)\right) \gtrsim R^{d-2} \int_{|\lambda| \leq R} |\widehat{h}(\lambda)|^2|\lambda|^2 {\rm d} m_d(\lambda) + R^d\int_{|\lambda| \geq R} |\widehat{h}(\lambda)|^2 {\rm d} m_d(\lambda)
	\end{equation*} 
	(cf. Theorem \ref{thm:upper_bound_for_variance}). Curiously, in Section \ref{sec:variance_of_number_of_points_inside_large_cube} we give a simple example that shows the matching upper bound for $\var\left(N_{\sf s}(h,R)\right)$ \emph{does not} hold.
	
	We now turn our focus to two natural classes of test functions where we are able to say more on the fluctuations.
	\subsection{Smooth linear statistics}
	Recall that the \emph{Sobolev space} $H^1(\bR^d)$ consists of functions $h\in L^2(\bR^d)$ having the distributional gradient $\nabla h$ such that $|\nabla h|\in L^2(\bR^d)$. Equivalently, the space $H^1(\bR^d)$ can be characterized by the condition 
	\[
	\int_{\bR^d} |\widehat{h}(\lambda)|^2\left(1 + |\lambda|^2\right) \dd m_d(\lambda) < \infty.
	\] 
	A good reference for basic facts about the Sobolev space is the book by H\"{o}rmander \cite[Chapter~7]{hormander}.
	\begin{theorem}(Smooth statistics)
		\label{thm:variance_for_sobolev_space}
		Suppose that $h\in \left(L^1\cap H^1\right)(\mathbb{R}^d)$. Then,
		\begin{equation*}
		\lim_{R\to\infty} \frac{{\normalfont \text{Var}}\left(N(h,R)\right)}{R^{d-2}} = \frac{a}{2}\int_{\mathbb{R}^d} |\nabla h(x)|^2 {\rm d} m_d(x).
		\end{equation*}
	\end{theorem}  
	Focusing on the planar case $d=2$, we see that the variance of smooth linear statistics tends to a non-zero limit as $R\to \infty$. A similar behavior for smooth linear statistics was observed also for eigenvalues of large Ginibre random matrix (that is, square matrices where all entries are i.i.d. complex Gaussians) in the work of Rider and Vir\'{a}g \cite{rider_virag}. We also mention that the limit in Theorem \ref{thm:variance_for_sobolev_space} was noticed in the work of Sodin and Tsirelson \cite{sodin_tsirelson} (see the introduction therein) whenever the test function $h\in C^2(\bR^2)$ has compact support. 
	
	It is worth mentioning that an analogue of Theorem \ref{thm:variance_for_sobolev_space} does not hold for the stationary counterpart $W_{\sf s}$. In Section \ref{sec:large_variance_for_function_in_Sobolev_space} we provide a function $g\in \left(L^1\cap H^1\right)(\mathbb{R}^d)$ for $d\ge2$ such that
	\[
	\limsup_{R\to \infty} \frac{\var\left(N_{\sf s}(g,R)\right)}{R^{d-2}} = +\infty.
	\]
	\subsection{Number of points in convex sets}
	Here we turn our attention to test functions which are indicators of convex sets with smooth boundary in $\bR^d$, $d\ge 2$. 
	\begin{theorem}
		\label{thm:variance_for_convex_sets}
		Suppose that $K\subset \mathbb{R}^d$ is a compact convex set such that $\partial K$ is a smooth closed manifold with nowhere vanishing Gaussian curvature. Then
		\begin{equation}
			\label{eq:limit_of_variance_in_thm_convex_set}
			\lim_{R\to\infty} \frac{{\normalfont \text{Var}}\left(\mathbf{n}(K_R)\right)}{R^{d-1}} = \sqrt{\frac{a}{2\pi}} \cdot \sigma_{d-1}(\partial K),
		\end{equation}
		where $\sigma_{d-1}(\partial K)$ is the surface area of $\partial K$. 
	\end{theorem}
	Theorem \ref{thm:variance_for_convex_sets} implies that the point process $W$ is \emph{hyperuniform}, in the sense introduced by Torquato and Stillinger \cite{torquato_stillinger} . That is, the fluctuations of the number of points that fall inside a nice convex domain grow like the surface area as the volume tends to infinity. In fact, the above growth rate falls into class \RomanNumeralCaps{1} hyperuniform point processes, the two other classes being of faster growth rate but still less than the volume of the convex set. For more details on fluctuations in hyperunifom systems see the survey by Ghosh and Lebowitz \cite{ghosh_lebowitz_survey} for mathematical results and the survey by Torquato \cite{torquato_hyperuniform_states} for the physics point of view.
	
	It is worth mentioning that the assumption on the Gaussian curvature is essential for \eqref{eq:limit_of_variance_in_thm_convex_set} to hold. In \cite{kim_torquato}, Kim and Torquato considered the case where
	\begin{equation}
		\label{eq:balls_in_l_p_norms}
		K = \bB^p \eqdef \left\{x=(x_1,\ldots,x_d) : |x_1|^{p} + \ldots |x_d|^p \leq 1 \right\}
	\end{equation}
	and observed numerically that for $d=2$ the variance of $\mathbf{n}(\bB_R^p)$ grows like $R^\gamma$, where $\gamma=\gamma(p)$ varies continuously between $1$ and $2$ as $p$ varies from $2$ to infinity. It is evident that $\bB^p$ is convex for all $p\ge 2$ and that $\partial \bB^p$ has non-vanishing curvature if and only if $p=2$, which is exactly the case covered by our Theorem \ref{thm:variance_for_convex_sets}.
	
	Finally, we mention the recent work of Adhikari, Ghosh and Lebowitz \cite{fluctuation_kartick} in which they study the asymptotic fluctuations in a certain class of hyperuniform systems, where the points of the process are fixed (non-random) and each point is assigned with a weight governed by an underlying mean-zero random field. Although the model in \cite{fluctuation_kartick} is not directly related to the perturbations of the lattice, they notice that the asymptotic of fluctuations in these processes depend on the shape of the growing domain, and is different when considering large balls as opposed to large cubes (see \cite[Section 1.7.3]{fluctuation_kartick} for more refined details). As already indicated in the introduction, our work highlights a similar phenomena.
	
	\subsection{Theorem \ref{thm:variance_for_convex_sets} and the work of G\'{a}cs and Sz\'{a}sz \cite{gacs_szasz}}
	A simple application of Theorem \ref{thm:variance_linear_statstic} and Fubini yields that
	\[
	\frac{\int_{[0,1]^d}\var\left(\mathbf{n}(K_R+x)\right){\rm d}m_d(x)}{R^{d-1}} = R\int_{\bR^d}|\widehat{\mathbf{1}_K}(\lambda)|^2 \left(1-e^{-2a\pi^2|\lambda|^2/R^2} \right) \dd m_d(\lambda),
	\] 
	see the proof of Theorem $\ref{thm:variance_linear_statstic_stat}$ for the details. Furthermore, in Section \ref{sec:convex_sets_with_smooth_boundary} we prove that
	\begin{align*}
		 &\lim_{R\to\infty} \frac{\int_{[0,1]^d}\var\left(\mathbf{n}(K_R+x)\right){\rm d}m_d(x)}{R^{d-1}} \\ &= \lim_{R\to\infty} R\int_{\bR^d}|\widehat{\mathbf{1}_K}(\lambda)|^2 \left(1-e^{-2a\pi^2|\lambda|^2/R^2} \right) \dd m_d(\lambda) = \sqrt{\frac{a}{2\pi}} \cdot \sigma_{d-1}(\partial K),
	\end{align*}
	provided that $K$ is a compact convex set (see Claim \ref{claim:leading_order_for_var_convex_sets}). The smoothness assumption on $\partial K$ allows us to show that the infinite sum in Theorem $\ref{thm:variance_linear_statstic_stat}$ does not contribute to the leading order asymptotic to $\var\left(\mathbf{n}_{\sf s}(K_R)\right)$. Still, the assumption of the non-vanishing curvature is essential, as in Section \ref{sec:variance_of_number_of_points_inside_large_cube} we show that if $K=[-\frac{1}{2},\frac{1}{2}]^d$ is the $d$-dimensional cube then $\var\left(\mathbf{n}_{\sf s}(K_R)\right)$ can be much larger than $R^{d-1}$. 
	
	In \cite[Theorem~1]{gacs_szasz}, it is proved that the limit
	$$
	\lim_{R\to\infty} \frac{\int_{[0,1]^d}\var\left(\mathbf{n}(K_R+x)\right){\rm d}m_d(x)}{R^{d-1}}
	$$ 
	exists when considering i.i.d. random perturbations of lattice points with almost an arbitrary distribution (see Section \ref{sec:convex_sets_with_smooth_boundary} for more precise details). The main motivation to their paper was a problem posed by D.R. Cox: determine the asymptotic behavior for the variance of number of displaced points (i.i.d. random perturbations of a lattice) contained inside a large convex set. See the introduction in \cite{gacs_szasz} for the exact formulation of Cox's problem. To quote from \cite{gacs_szasz}: \emph{``The question becomes more tractable when replacing the variance with its value averaged over the unit cube.."}. Our Theorem \ref{thm:variance_for_convex_sets} shows that in the case of Gaussian perturbations the extra averaging is not necessary, as long as the boundary has non-vanishing Gaussian curvature, while the example of the unit cube treated in Section \ref{sec:variance_of_number_of_points_inside_large_cube} shows that the curvature assumption is needed for that remark.
	
	\section{Mean of the linear statistics}
	\label{sec:mean_for_linear_statistics}
	
	As usual, we write $e(x) = \exp\left(2\pi i x\right)$. Throughout we write $N(h) = N(h,1)$. Recall that the the \emph{characteristic function} of the Gaussian vector $\xi_0$ is given by
	\begin{equation*}
	\label{eq:d_gaussian_measure_fourier_transform}
	\E\left[e(-\inner{\xi_0}{\lambda})\right] = \widehat{\phi_a}(\lambda) = \exp\left( - a \pi^2 |\lambda|^2\right).
	\end{equation*}
	The convolution of two functions $f,g\in L^1\left(\mathbb{R}^d\right)$ is given by
	\[
	\left(f*g\right)\left(y\right) = \int_{\mathbb{R}^d} f(x) g(y-x) {\rm{d}} m_{d}(x),
	\]
	$f\ast g \in L^1(\bR^d)$ and $\widehat{\left(f*g\right)} = \widehat{f}\cdot \widehat{g}$. Finally, we quote a version of the \emph{Poisson summation formula} which we will use several times. 
	\begin{proposition}[{\cite[Theorem~2.1]{applebaum}}]
		\label{prop:poisson_summation_formula}
		Suppose that $f\in L^1\left(\bR^d\right)$ is continuous such that $\sum_{m\in\bZ^d}|\widehat{f}(m)|<\infty$. Assume further that the \emph{periodization} of $f$ $$(\mathcal{P}f)(x)\eqdef \sum_{n\in \bZ^d}f(x-n)$$ converges absolutely and uniformly for all $x\in \bR^d$, then  
		$$
		\sum_{n\in \bZ^d} f(n) = \sum_{m\in \bZ^d} \widehat{f}(m).
		$$
	\end{proposition}	
	\begin{lemma}
		\label{lemma:mean_linear_statstic}
		For any $h\in\left(L^1\cap L^2\right)\left(\mathbb{R}^d\right)$ we have,
		\begin{equation*}
		\label{eq:mean_linear_statistic}
		\E\left[N(h,R)\right] = R^d\sum_{m\in \bZ^d} e^{-a\pi^2|m|^2} \widehat{h}(Rm) .
		\end{equation*}
	\end{lemma}
	\begin{proof}
		By the scaling relation $\widehat{h(\cdot/R)} = R^d \widehat{h}(R\cdot)$ it is enough to show the equality holds only for the case $R=1$. We Set $\ti{h} = h\ast \phi_a$ where $\phi_a$ is the Gaussian function (\ref{eq:d_gaussian_measure}). Clearly $\ti{h}\in \left(L^1\cap C^\infty\right)(\bR^d)$, so in order to apply Proposition \ref{prop:poisson_summation_formula} we are left to show that the periodization converges. Indeed, we may apply the dominated convergence theorem as
		\begin{align*}
		\left|(\mathcal{P}\ti{h})(x)\right| &\leq \sum_{n\in \bZ^d} \left\{ \int_{\mathbb{R}^d} \left|h(z)\right|\phi_a(x-n-z) {\rm d}m_d(z) \right\} \\  &=  \int_{\mathbb{R}^d}|h(z)| \left\{\sum_{n\in \bZ^d}\phi_a(x-z-n) \right\} {\rm d}m_d(z) \lesssim \int_{\mathbb{R}^d} \left|h(z)\right|{\rm d}m_d(z).
		\end{align*}
		It remains to observe that for all $n\in \bZ^d$,
		$$
		\E\left[h(n+ \xi_n)\right] = \int_{\mathbb{R}^d} h(n+x) \phi_a(x) {\rm d} m_d(x) = \ti{h}(n).
		$$ 
		Altogether, we apply Poisson summation formula and Fubini to get get
		\begin{align*}
			\E\left[N(h)\right] &= \sum_{n\in \bZ^d} \ti{h}(n) \\ &= \sum_{m\in \bZ^d} \widehat{\left(h\ast\phi_a\right)}(m) = \sum_{m\in \bZ^d} e^{-a\pi^2|m|^2}\widehat{h}(m).
		\end{align*}
	\end{proof}
	
	\begin{corollary}
	\label{cor:mean_linear_statistics_first_order}
	For any $h\in\left(L^1\cap L^2\right)\left(\mathbb{R}^d\right)$ we have
	\begin{equation*}
	\lim_{R\rightarrow\infty} \frac{\E\left[N(h,R)\right]}{R^d} = \int_{\mathbb{R}^d} h \, {\rm d}m_d.
	\end{equation*}
	\end{corollary}
	\begin{proof}
		Follows immediately from Lemma \ref{lemma:mean_linear_statstic} combined with the Riemann-Lebesgue lemma \cite[Proposition~2.2.17]{grafakos} and the dominated convergence theorem.
	\end{proof}
	 Another simple consequence of Lemma \ref{lemma:mean_linear_statstic} is a formula for the mean in the translation invariant case.
	\begin{corollary}
		\label{cor:mean_linear_statistics_stat}
		For all $h\in \left(L^1\cap L^2\right)\left(\bR^d\right)$ we have
		\[
		\E\left[N_{\sf s}(h,R)\right] = R^d \int_{\bR^d} h \, {\rm d}m_d.
		\]
	\end{corollary} 
	\begin{proof}
		Again by scaling we may prove only for $R=1$. Set $N_{\sf s}(h) = N_{\sf s}(h,1)$. Using that $\widehat{h(\cdot + \zeta)} = e(\inner{\cdot}{\zeta}) \widehat{h}(\cdot)$ we can apply the law of total expectation (see \cite[eq.~(4.1.5)]{durrett}) and observe that
		\begin{align*}
			\E\left[N_{\sf s}(h)\right] &= \E\left[\E\left[N_{\sf s}(h)\mid \zeta \right]\right] \\ &= \E\left[\sum_{m\in \bZ^d} e^{-a\pi^2|m|^2} e(\inner{m}{\zeta}) \widehat{h}(m) \right] \\ &= \sum_{m\in \bZ^d}  e^{-a\pi^2|m|^2} \widehat{h}(m) \E\left[e(\inner{m}{\zeta})\right]
		\end{align*}
		where the exchange of sum and expectation is valid since the sum is uniformly and absolutely convergent. It remains to observe that
		\begin{equation}
			\label{eq:char_function_of_uniform_variable_on_lattice_pts}
			\E\left[e(\inner{m}{\zeta})\right] = \int_{[0,1]^d} e(\inner{m}{x}) \dd m_d(x) = \begin{cases}
			1 & m=0, \\ 0 & m\in \bZ^d\setminus\{0\}.
			\end{cases}
		\end{equation}
	\end{proof}
	Suppose now that $K\subset \mathbb{R}^d$ is a compact convex set such that $\partial K$ is a smooth closed manifold with nowhere vanishing Gaussian curvature. We have the following bound on the decay of the Fourier transform of $\mathbf{1}_K$, given as
	\begin{equation}
	\label{eq:bound_on_fourier_transform_of_smooth_set}
	|\widehat{\mathbf{1}_K}(\lambda)| \lesssim \left(1+|\lambda|\right)^{-(d+1)/2}.
	\end{equation} 
	Here the implicit constant depends only on the Gaussian curvature of $\partial K$, see \cite[Corollary~7.7.15]{hormander}. In fact, (\ref{eq:bound_on_fourier_transform_of_smooth_set}) is a consequence of the more general bound (\ref{eq:bound_on_fourier_transfom_of_manifold}) which we use in Section \ref{sec:convex_sets_with_smooth_boundary}. For such sets $K$, we give an upper bound on the remainder term in Corollary \ref{cor:mean_linear_statistics_first_order}.
		\begin{lemma}
		\label{lemma:remainder_term_in_the_mean_convex_sets}
		Suppose that $K\subset \mathbb{R}^d$ is a compact convex set such that $\partial K$ is a smooth closed manifold with nowhere vanishing Gaussian curvature. Then
		\[
		\left|\E\left[\mathbf{n}\left(K_R\right)\right] - m_d(K) R^d\right| \lesssim R^{(d-1)/2}
		\] 
	\end{lemma}
	
	\begin{proof}
		Combining Lemma \ref{lemma:mean_linear_statstic} with the bound (\ref{eq:bound_on_fourier_transform_of_smooth_set}) we get that
		\begin{align*}
			\left|\E\left[\mathbf{n}\left(K_R\right)\right] - m_d(K) R^d\right| & \leq R^d \sum_{m\in \bZ^d \setminus \{0\}} e^{-a\pi^2|m|^2} |\widehat{\mathbf{1}_K}(Rm)| \\ & \lesssim R^{(d-1)/2} \sum_{m\in \bZ^d \setminus \{0\}} e^{-a\pi^2|m|^2} |m|^{-(d+1)/2}
		\end{align*}
		and the lemma follows as $\sum_{m\in \bZ^d \setminus \{0\}} e^{-a\pi^2|m|^2} |m|^{-(d+1)/2} < \infty$ for all $a>0$.
	\end{proof}
	\subsection{Mean number of points from $W$ inside a large cube}
	\label{sec:mean_number_of_points_in_large_cube}
	We give a simple example to show that in Lemma \ref{lemma:remainder_term_in_the_mean_convex_sets} the assumption on the Gaussian curvature is necessary. We will do so by examining the case $Q \eqdef [-\frac{1}{2},\frac{1}{2}]^d$. By setting $\lambda=\left(\lambda_1,\ldots,\lambda_d\right)$ we can compute the Fourier transform of $\mathbf{1}_{Q}$ as
	\begin{equation}
	\label{eq:fourier_transform_of_cube}
	\widehat{\mathbf{1}_{Q}} (\lambda) = \int_{[-\frac{1}{2},\frac{1}{2}]^d} e\left(-\inner{\lambda}{x}\right) \dd m_d(x) = \prod_{j=1}^{d} \text{sinc}(\pi \lambda_j), 
	\end{equation}
	where,
	\[
	\text{sinc}(x) \eqdef \begin{cases}
	\left(\sin x\right)/x & x\not= 0, \\ 1 &x=0.
	\end{cases}
	\]
	It follows from Lemma \ref{lemma:mean_linear_statstic} that
	\begin{align}
		\label{eq:mean_number_of_points_in_large_cube_sum}
		\E\left[\mathbf{n}(Q_R)\right] &= R^d \sum_{m\in \bZ^d} e^{-a\pi^2|m|^2} \left(\prod_{j=1}^{d} \text{sinc}(\pi R m_j)\right).
	\end{align}
	We split the sum in (\ref{eq:mean_number_of_points_in_large_cube_sum}) according to the number of zero entries in the vector $m = (m_1,\ldots,m_d)\in \bZ^d$ and obtain that
	\begin{align*}
		\E\left[\mathbf{n}(Q_R)\right] &= R^d + R^d\Bigg(\sum_{\#\{j : \, m_j\not= 0\} = 1} e^{-a\pi^2|m|^2} \prod_{j=1}^{d} \text{sinc}(\pi R m_j) \Bigg) + \mathcal{O}\left(R^{d-2}\right) \\ &=R^d + R^{d-1} \left(2d  \sum_{\ell=1}^{\infty} e^{-a\pi^2 \ell^2} \frac{\sin(\pi R\ell)}{\pi\ell}\right) + \mathcal{O}\left(R^{d-2}\right).
	\end{align*}
	By looking at a subsequence of $R = j+\frac{1}{2}$ for $j\in \bZ_{\ge 0}$ we immediately get that
	\[
	\limsup_{R\to \infty} \frac{\E\left[\mathbf{n}(Q_R)\right] - R^d}{R^{d-1}} > 0.
	\]
	\section{Fluctuations of linear statistics}
	\label{sec:fluctuations_of_linear_statistics}
	\begin{proof}[Proof of Theorem \ref{thm:variance_linear_statstic}]
		We first prove the equality for $R=1$. Denote for the moment $\phi = \phi_a$ and set $H(z) = \var\left(h(z+\xi_0)\right)$. Since the $\xi_n$'s are independent we have
		\[
		\var\left(N(h)\right) = \sum_{n\in \bZ^d} H(n).
		\]
		Recall that $\ti{h} = h\ast \phi$. By the definition of the variance,
		\begin{align*}
		H(x) &= \E\left[h^2(x+\xi_0)\right] - \left(\E\left[h(x+\xi_0)\right]\right)^2 \\ &= \left(h^2\ast \phi\right)(x) - \left(h*\phi\right)^2(x) = \ti{h^2}(x) - (\ti{h})^2(x).
		\end{align*}
		The Cauchy-Schwarz inequality combined with our assumption $h\in L^2\left(\mathbb{R}^d\right)$ implies that $H\in \left(L^1\cap C^\infty\right)\left(\mathbb{R}^d\right)$. As we wish to apply Proposition \ref{prop:poisson_summation_formula}, we need to give a uniform bound for the periodization. By repeating the same argument as in the proof of Lemma \ref{lemma:mean_linear_statstic} we see that $(\mathcal{P}\ti{h^2})(x)$ converges absolutely and uniformly for all $x\in \bR^d$. For the second term, observe that
		\begin{align*}
			\left|(\mathcal{P}(\ti{h})^2)(x)\right| \leq \iint_{\bR^d\times \bR^d} |h(z)h(w)|&\left\{\sum_{n\in\bZ^d} \phi(x-n-z)\phi(x-n-w)\right\} \dd m_d(z) \dd m_d(w) \\ &\lesssim \left(\int_{\bR^d} |h(z)| \dd m_d(z)\right)^2,
		\end{align*} 
		and so, by the dominated convergence theorem, $(\mathcal{P}H)(x)$ is absolutely and uniformly convergent for all $x\in \bR^d$. Finally, the Fourier transform of $H$ is given by
		\begin{align*}
		\widehat{H}(\eta) &= \left(\widehat{h^2}\widehat{\phi}\right) (\eta) - \left[\left(\widehat{h}\widehat{\phi}\right)\ast\left(\widehat{h}\widehat{\phi}\right)\right] (\eta) \\ & = \left[\left(\widehat{h}\ast \widehat{h}\right) \widehat{\phi}\right](\eta) - \left[\left(\widehat{h}\widehat{\phi}\right)\ast\left(\widehat{h}\widehat{\phi}\right)\right] (\eta) \\ & =\int_{\mathbb{R}^d}\widehat{h}(\lambda)\widehat{h}(\eta-\lambda)\left(e^{-a\pi^2|\eta|^2} - e^{-a\pi^2|\lambda|^2}e^{-a\pi^2|\eta-\lambda|^2}\right){\rm d} m_{d}(\lambda),
		\end{align*}
		and by Proposition \ref{prop:poisson_summation_formula} we get
		\begin{align*}
		\var\left(N(h)\right) &= \sum_{n\in \bZ^d} H(n) =  \sum_{m\in \bZ^d} \widehat{H}(m)\\ &= \sum_{m\in \bZ^d} \left\{\int_{\mathbb{R}^d}\widehat{h}(\lambda)\widehat{h}(m-\lambda)\left(e^{-a\pi^2|m|^2} - e^{-a\pi^2|\lambda|^2}e^{-a\pi^2|m-\lambda|^2}\right){\rm d} m_{d}(\lambda)\right\}.
		\end{align*}
		To get the result for general $R$ we use the scaling property of the Fourier transform and a change of variables $\mu=R\lambda$. 
	\end{proof}
	As a first corollary of Theorem \ref{thm:variance_linear_statstic}, we prove an upper bound on $\var\left(N(h,R)\right)$ valid for all test functions $h\in (L^1\cap L^2)(\mathbb{R}^d)$.
	\begin{proof}[Proof of Theorem \ref{thm:upper_bound_for_variance}]
		Throughout this proof we denote by $C_a>0$ a constant that depends only on the parameter $a>0$ (and may change from line to line). Using the scaling relation $\widehat{h(\cdot/R)} = R^d \widehat{h}(R\cdot)$ and the change of variables $\mu=R\lambda$, it is enough to prove that
		\begin{equation*}
			\var\left(N(h)\right) \leq C_a\left(\int_{|\lambda|\leq 1} |\widehat{h}(\lambda)|^2|\lambda|^2 {\rm d} m_d(\lambda) + \int_{|\lambda|\ge 1} |\widehat{h}(\lambda)|^2 {\rm d} m_d(\lambda) \right).
		\end{equation*}
		For every point $m\in \mathbb{Z}^d$, we put
		\begin{equation}
			\label{eq:def_of_A_m_h_thm_3}
			A_m(h) = \int_{\mathbb{R}^d} \widehat{h}(\lambda)\widehat{h}(m-\lambda)\left(e^{-a\pi^2|m|^2}-e^{-a\pi^2|\lambda|^2-a\pi^2|m-\lambda|^2}\right) {\rm d}m_d(\lambda).
		\end{equation}
		By the inequality $1-e^{-x} \leq \min\{x,2\}$ for $x>0$, we can bound the term $m=0$ by
		\[
		A_0(h) \leq 2a\pi^2\int_{|\lambda|\leq 1} |\widehat{h}(\lambda)|^2|\lambda|^2{\rm d}m_d(\lambda) + 2\int_{|\lambda|>1}|\widehat{h}(\lambda)|^2{\rm d}m_d(\lambda). 
		\]
		Fix $m\in \mathbb{Z}^d\setminus\{0\}$. We split (\ref{eq:def_of_A_m_h_thm_3}) into three parts,
		\begin{equation}
			\label{eq:variance_upper_bound_split_into_3}
			A_m(h)=\left(\int_{\RomanNumeralCaps{1}} + \int_{\RomanNumeralCaps{2}} + \int_{\RomanNumeralCaps{3}}\right) \widehat{h}(\lambda)\widehat{h}(m-\lambda)\left(e^{-a\pi^2|m|^2}-e^{-a\pi^2|\lambda|^2-a\pi^2|m-\lambda|^2}\right) {\rm d}m_d(\lambda)
		\end{equation}
		where,
		\[
		{\RomanNumeralCaps{1}} = \left\{|\lambda|\leq 1/2 \right\}, \quad {\RomanNumeralCaps{2}} = \left\{|\lambda-m|\leq 1/2 \right\}, \quad  {\RomanNumeralCaps{3}} = \mathbb{R}^d\setminus\left({\RomanNumeralCaps{1}}\cup{ \RomanNumeralCaps{2}}\right).
		\]
		Turning to bound the first integral in (\ref{eq:variance_upper_bound_split_into_3}), we use again $1-e^{-x}\leq x$ to get,
		\begin{align*}
			\Big|\int_{\RomanNumeralCaps{1}} \widehat{h}(\lambda)\widehat{h}(m-\lambda)&\left(e^{-a\pi^2|m|^2}-e^{-a\pi^2|\lambda|^2-a\pi^2|m-\lambda|^2}\right) {\rm d}m_d(\lambda)\Big| \\ &\leq e^{-a\pi^2|m|^2}\int_{|\lambda|\leq 1/2} |\widehat{h}(\lambda)\widehat{h}(m-\lambda)| \left|1-e^{-2a\pi^2\left(|\lambda|^2+\inner{m}{\lambda}\right)}\right|{\rm d}m_d(\lambda) \\ &\leq C_a e^{-a\pi^2|m|^2}\int_{|\lambda|\leq 1/2} |\widehat{h}(\lambda)\widehat{h}(m-\lambda)|\left(|\lambda|^2 + \inner{\lambda}{m}\right) {\rm d}m_d(\lambda) \\ &\leq C_a e^{-a\pi^2|m|^2} |m| \int_{|\lambda|\leq 1/2} |\widehat{h}(\lambda)\widehat{h}(m-\lambda)| |\lambda| {\rm d}m_d(\lambda).
		\end{align*}
		Continuing, by the Cauchy-Schwarz inequality,
		\begin{align}
			\label{eq:variance_upper_bound_split_into_3_1}
			\Big|&\int_{\RomanNumeralCaps{1}} \widehat{h}(\lambda)\widehat{h}(m-\lambda)\left(e^{-a\pi^2|m|^2}-e^{-a\pi^2|\lambda|^2-a\pi^2|m-\lambda|^2}\right) {\rm d}m_d(\lambda)\Big|  \\ &\leq C_a e^{-a\pi^2|m|^2} |m| \left(\int_{|\lambda|\leq 1/2} |\widehat{h}(\lambda)|^2 |\lambda|^2 {\rm d}m_d(\lambda)\right)^{1/2} \left(\int_{|\lambda|\leq 1/2}|\widehat{h}(m-\lambda)|^2 {\rm d}m_d(\lambda)\right)^{1/2} \nonumber \\ &\leq C_a e^{-a\pi^2|m|^2} |m| \left(\int_{|\lambda|\leq 1/2} |\widehat{h}(\lambda)|^2 |\lambda|^2 {\rm d}m_d(\lambda) + \int_{|\lambda|> 1/2} |\widehat{h}(\lambda)|^2 {\rm d}m_d(\lambda) \right) \nonumber
		\end{align}
		where in the last inequality we used the fact that $|m|\ge 1$. By the change of variables $\mu=\lambda + m $ we see that the first and second integral in (\ref{eq:variance_upper_bound_split_into_3}) are equal, and we obtain that
		\begin{align}
			\label{eq:variance_upper_bound_split_into_3_2}
			\Big|&\int_{\RomanNumeralCaps{2}} \widehat{h}(\lambda)\widehat{h}(m-\lambda)\left(e^{-a\pi^2|m|^2}-e^{-a\pi^2|\lambda|^2-a\pi^2|m-\lambda|^2}\right) {\rm d}m_d(\lambda)\Big| \\ &\leq C_a e^{-a\pi^2|m|^2} |m| \left(\int_{|\lambda|\leq 1/2} |\widehat{h}(\lambda)|^2 |\lambda|^2 {\rm d}m_d(\lambda) + \int_{|\lambda|> 1/2} |\widehat{h}(\lambda)|^2 {\rm d}m_d(\lambda) \right). \nonumber
		\end{align} 
		It remains to bound the integral over the domain $\RomanNumeralCaps{3}$. We use the Cauchy-Schwarz inequality once more,
		\begin{align}
			\label{eq:variance_upper_bound_split_into_3_3}
			\Big|\int_{\RomanNumeralCaps{3}} &\widehat{h}(\lambda)\widehat{h}(m-\lambda)\left(e^{-a\pi^2|m|^2}-e^{-a\pi^2|\lambda|^2-a\pi^2|m-\lambda|^2}\right) {\rm d}m_d(\lambda)\Big| \\ &\leq 2e^{-a\pi^2|m|^2/2} \int_{\RomanNumeralCaps{3}} |\widehat{h}(\lambda)\widehat{h}(m-\lambda)| {\rm d}m_d(\lambda) \nonumber \\ & \leq 2e^{-a\pi^2|m|^2/2} \int_{|\lambda|> 1/2} |\widehat{h}(\lambda)|^2 {\rm d}m_d(\lambda). \nonumber
		\end{align}
		Plugging (\ref{eq:variance_upper_bound_split_into_3_1}), (\ref{eq:variance_upper_bound_split_into_3_2}) and (\ref{eq:variance_upper_bound_split_into_3_3}) into relation (\ref{eq:variance_upper_bound_split_into_3}) yields the upper bound
		\[
		|A_m(h)| \leq C_a e^{-a\pi^2|m|^2/2}|m| \left(\int_{|\lambda|\leq 1/2} |\widehat{h}(\lambda)|^2 |\lambda|^2 {\rm d}m_d(\lambda) + \int_{|\lambda|> 1/2} |\widehat{h}(\lambda)|^2 {\rm d}m_d(\lambda) \right)
		\]
		which, together with Theorem \ref{thm:variance_linear_statstic} implies that
		\begin{align*}
			\var\left(N(h)\right) &= \sum_{m\in \mathbb{Z}^d} A_m(h) \\ &\leq C_a \sum_{m\in \mathbb{Z}^d} e^{-a\pi^2|m|^2/2}|m| \left(\int_{|\lambda|\leq 1/2} |\widehat{h}(\lambda)|^2 |\lambda|^2 {\rm d}m_d(\lambda) + \int_{|\lambda|> 1/2} |\widehat{h}(\lambda)|^2 {\rm d}m_d(\lambda) \right) \\ &= C_a \left(\int_{|\lambda|\leq 1/2} |\widehat{h}(\lambda)|^2 |\lambda|^2 {\rm d}m_d(\lambda) + \int_{|\lambda|> 1/2} |\widehat{h}(\lambda)|^2 {\rm d}m_d(\lambda) \right).
		\end{align*}
	\end{proof}
	Another application of Theorem \ref{thm:variance_linear_statstic} is a similar formula for the variance of linear statistics of the translation-invariant process $W_{\sf s}$.
		\begin{proof}[Proof of Theorem $\ref{thm:variance_linear_statstic_stat}$]
			As before, we only prove for $R=1$ and then use scaling to get the desired result. By the law of total variance \cite[Exercise~4.1.7, follows easily from Theorem~4.1.15 therin]{durrett} we know that
			\begin{equation}
				\label{eq:conditional_variance_formula}
				\var\left(N_{\sf s}\left(h\right)\right) = \E\left[\var\left(N_{\sf s}\left(h\right)\mid \zeta \right) \right] + \var\left(\E\left[N_{\sf s}\left(h\right) \mid \zeta\right]\right).
			\end{equation}
			We compute each of the terms in (\ref{eq:conditional_variance_formula}) separately. Indeed, since the random vector $\zeta$ is independent of the sequence $\{\xi_n\}$, we apply Theorem \ref{thm:variance_linear_statstic} and obtain
			\begin{align*}
				\var&\left(N_{\sf s}\left(h\right)\mid \zeta \right) \\ &= \sum_{m\in \bZ^d} e\left(\inner{m}{\zeta}\right)\left\{\int_{\mathbb{R}^d}\widehat{h}(\lambda)\widehat{h}(m-\lambda)\left(e^{-a\pi^2|m|^2} - e^{-a\pi^2|\lambda|^2}e^{-a\pi^2|m-\lambda|^2}\right){\rm d} m_{d}(\lambda)\right\}.
			\end{align*}
			By $(\ref{eq:char_function_of_uniform_variable_on_lattice_pts})$ we obtain that
			\[
			\E\left[\var\left(N_{\sf s}\left(h\right)\mid \zeta \right) \right] = \int_{\mathbb{R}^d} |\widehat{h}(\lambda)|^2\left(1-e^{-2a\pi^2|\lambda|^2}\right) {\rm d}m_d(\lambda).
			\]
			To compute the second term of (\ref{eq:conditional_variance_formula}) we use Lemma \ref{lemma:mean_linear_statstic} and Corollary \ref{cor:mean_linear_statistics_stat} to see that
			\begin{align*}
				\var\left(\E\left[N_{\sf s}\left(h\right) \mid \zeta\right]\right) &= \E\left(\sum_{m\in \bZ^d}e\left(\inner{m}{\eta}\right)e^{-a\pi^2|m|^2}\widehat{h}(m)\right)^2 - \widehat{h}(0)^2 \\ &= \sum_{m,m^\prime \in \bZ^d} \E\left[e\left(\inner{m-m^\prime}{\zeta}\right)\right] e^{-a\pi^2(|m|^2 + |m^\prime|^2)}\widehat{h}(m)\overline{\widehat{h}(m^\prime)} - \widehat{h}(0)^2 \\ &\stackrel{(\ref{eq:char_function_of_uniform_variable_on_lattice_pts})}{=} \sum_{m\in \bZ^d \setminus \{0\}} e^{-2a\pi^2|m|^2} |\widehat{h}(m)|^2.
			\end{align*}
			Plugging into relation (\ref{eq:conditional_variance_formula}) yields the desired result.
		\end{proof}
	As a simple consequence of Theorem $\ref{thm:variance_linear_statstic_stat}$, we get a lower bound for the fluctuations for linear statistics of $W_{\sf s}$ as
	\begin{align*}
	{\normalfont \text{Var}}\left(N_{\sf s}(h,R)\right) &\gtrsim R^d\int_{\mathbb{R}^d} |\widehat{h}(\lambda)|^2\left(1-e^{-2a\pi^2|\lambda|^2/R^2}\right) {\rm d}m_d(\lambda) \\ & \gtrsim R^{d-2} \int_{|\lambda|\leq R} |\widehat{h}(\lambda)|^2 |\lambda|^2 \dd m_d(\lambda) + R^d\int_{|\lambda|\ge R} |\widehat{h}(\lambda)|^2 \dd m_d(\lambda).
	\end{align*}
	To show that the corresponding upper bound does not hold (in contrary to Theorem \ref{thm:upper_bound_for_variance}) we have the following simple example.
	\subsection{Variance of the number of points of $W_{\sf s}$ inside a large cube}
	\label{sec:variance_of_number_of_points_inside_large_cube}
	We assume here that $d\ge 2$. Recall that $\mathbf{n}_{\sf s}(Q_R) = N_{\sf s}(\mathbf{1}_{Q},R)$ and that $Q=[-\frac{1}{2},\frac{1}{2}]^d$. Recall from (\ref{eq:fourier_transform_of_cube}) that
	\begin{equation*}
		\widehat{\mathbf{1}_{Q}} (\lambda) = \prod_{j=1}^{d} \text{sinc}(\pi\lambda_j).
	\end{equation*}
	In Section \ref{sec:convex_sets_with_smooth_boundary} we show that
	\[
	R^d\int_{\mathbb{R}^d} |\widehat{\mathbf{1}_{Q}}(\lambda)|^2\left(1-e^{-2a\pi^2|\lambda|^2/R^2}\right) {\rm d}m_d(\lambda) = R^{d-1}\left(\sqrt{\frac{a}{2\pi}}2^d + o(1)\right),
	\]
	as $R\to\infty$ (see Claim \ref{claim:leading_order_for_var_convex_sets}). Still, the infinite sum in Theorem $\ref{thm:variance_linear_statstic_stat}$ can be much larger than $R^{d-1}$ for this particular choice of test function. Indeed, by summing only over the sub-lattice 
	\begin{equation}
		\label{eq:definition_of_the_sublattice}
		\mathcal{Z} \eqdef \{\left(j,0,\ldots,0\right)\in \bZ^d \mid j\in \bZ\},
	\end{equation} 
	we see that 
	\begin{align*}
		\var\left(\mathbf{n}_{{\sf s}} (Q_R)\right) & \ge R^{2d} \sum_{m\in\mathcal{Z}} e^{-2a\pi^2|m|^2} |\widehat{\mathbf{1}_{Q}}(Rm)|^2 \\ &\ge 2R^{2d} \sum_{j=1}^\infty e^{-2a\pi^2j^2} \left(\frac{\sin(\pi R j)}{\pi R j }\right)^2  \gtrsim R^{2(d-1)} \sin^2(\pi R) e^{-2a\pi^2} .
	\end{align*}
	And so, for $d\ge 2$ we finally get that
	\[
	\limsup_{R\to \infty} \frac{\var\left(\mathbf{n}_{{\sf s}} (Q_R)\right)}{R^{d-1}} = +\infty
	\] 
	which implies that the upper bound (\ref{eq:upper_bound_for_variance_general_statistic}) does not hold if we replace $N$ by $N_{\sf s}$.
	\section{Smooth linear statistics}
	\label{sec:smooth_linear_statistics}
	In this section we give the proof of Theorem \ref{thm:variance_for_sobolev_space}. Recall that for $h\in\left(L^1\cap L^2\right)(\mathbb{R}^d)$, Theorem \ref{thm:variance_linear_statstic} asserts that
	\begin{align}
		\label{eq:variance_as_a_sum_after_scaling}
		&\var\left(N(h,R)\right) = \sum_{m\in \mathbb{Z}^d} A_m(h,R) \qquad \text{where,} \\ \label{eq:the_m_term_in_varince_sum_with_scale}
		 A_m(h,R) = R^d\int_{\mathbb{R}^d} &\widehat{h}(\lambda)\widehat{h}(Rm-\lambda)\left(e^{-a\pi^2|m|^2} - e^{-a\pi^2|\lambda|^2/R^2}e^{-a\pi^2|Rm-\lambda|^2/R^2}\right){\rm d} m_{d}(\lambda).
	\end{align}
	The strategy for the proof is to show that the term $A_0(h,R)$ dominates the rest on the sum $\sum_{m\not=0} A_m(h,R)$. As before, we denote by $C_a>0$ an arbitrary constant that depends only on $a>0$.
	
	Recall that for a function $f\in\left(L^1 \cap H^1\right) \left(\mathbb{R}^d\right)$ we have the following identity
	\begin{equation}
	\label{eq:norm_of_gradiant_sobolev_space}
	4\pi^2 \int_{\mathbb{R}^{d}}|\widehat{h}(\lambda)|^2 |\lambda|^2 {\rm d} m_{d}(\lambda) = \int_{\mathbb{R}^d} |\nabla h(x)|^2 {\rm d} m_{d}(x).
	\end{equation} 
	See for instance \cite[Theorem~8.22]{folland}.
	\begin{claim}
		\label{claim:bound_on_remainder_variance_sobolev}
		Let $h\in \left(L^1\cap H^1\right)\left(\mathbb{R}^d\right)$ and let $A_m(h,R)$ be given by (\ref{eq:the_m_term_in_varince_sum_with_scale}). Then for any $m\in \mathbb{Z}^d \setminus \{0\}$
		\[
		|A_m(h,R)| \leq C_a e^{-a\pi^2|m|^2/2} |m| E_h(R),
		\]
		where $E_h(R)$ depends only on $h$ and $R$ and satisfies $E_h(R) = o(R^{d-2})$ as $R\to \infty$.
	\end{claim}
	\begin{proof}
		Fix $m\in \mathbb{Z}^d \setminus \{0\}$ and split (\ref{eq:the_m_term_in_varince_sum_with_scale}) into three parts:
		\begin{equation}
		\label{eq:variance_in_sobolev_split_remainder_integral}
		A_m(h,R) = \left(\int_{\RomanNumeralCaps{1}^\prime} +\int_{\RomanNumeralCaps{2}^\prime} +\int_{\RomanNumeralCaps{3}^\prime} \right) (\cdots) {\rm d}m_d(\lambda),
		\end{equation}
		where,
		\[
		{\RomanNumeralCaps{1}}^\prime = \left\{|\lambda| \leq R/2 \right\}, \quad {\RomanNumeralCaps{2}}^\prime = \left\{|\lambda-Rm| \leq R/2 \right\}, \quad \RomanNumeralCaps{3}^\prime = \mathbb{R}^d \setminus \left(\RomanNumeralCaps{1}^\prime\cup \RomanNumeralCaps{2}^\prime\right).
		\]
		We start by bounding the first integral in (\ref{eq:variance_in_sobolev_split_remainder_integral}). By the triangle inequality
		\begin{align}
		\label{eq:first_bound_for_first_integral_sobolev_variance}
		\Big|&R^d\int_{\RomanNumeralCaps{1}^\prime} \widehat{h}(\lambda)\widehat{h}(Rm-\lambda)\left(e^{-a\pi^2|m|^2} - e^{-a\pi^2|\lambda|^2/R^2}e^{-a\pi^2|Rm-\lambda|^2/R^2}\right){\rm d} m_{d}(\lambda)\Big| \\ & \leq e^{-a\pi^2|m|^2}R^{d} \int_{|\lambda|\leq R/2} |\widehat{h}(\lambda)\widehat{h}(Rm-\lambda)||1-e^{-2a\pi^2|\lambda|^2/R^2 - 2a\pi^2 \inner{\lambda}{Rm}/R^2}|{\rm d} m_{d}(\lambda) \nonumber \\ &\leq C_a e^{-a\pi^2|m|^2} \bigg(R^{d-2} \int_{|\lambda|\leq R/2} |\widehat{h}(\lambda)\widehat{h}(Rm-\lambda)| |\lambda|^2 {\rm d} m_{d}(\lambda) \nonumber \\ & \qquad \qquad \qquad \qquad + R^{d-1}|m| \int_{|\lambda|\leq R/2} |\widehat{h}(\lambda)\widehat{h}(Rm-\lambda)| |\lambda| {\rm d} m_{d}(\lambda)\bigg) \nonumber \\ &\eqdef C_a e^{-a\pi^2|m|^2} \left({\sf E_1} + {\sf E_2}\right). \nonumber
		\end{align}
		We turn to bound ${\sf E}_i$ for $i=1,2$. Since $h\in H^1\left(\mathbb{R}^d\right)$ we can use Cauchy-Schwarz inequality and get that
		\begin{align}
		\label{eq:variance_of_sobolev_bound_E_1}
		{\sf E}_1 &= R^{d-2} \int_{|\lambda|\leq R/2} |\widehat{h}(\lambda)\widehat{h}(Rm-\lambda)| |\lambda|^2 {\rm d} m_{d}(\lambda) \\ & \leq R^{d-2} \left(\int_{|\lambda|\leq R/2} |\widehat{h}(\lambda)|^2|\lambda|^2{\rm d}m_d(\lambda)\right)^{1/2} \left(\int_{|\lambda|\leq R/2} |\widehat{h}(Rm-\lambda)|^2|\lambda|^2{\rm d}m_d(\lambda)\right)^{1/2} \nonumber \\ & \leq R^{d-2}\left(\int_{\bR^d} |\widehat{h}(\lambda)|^2|\lambda|^2{\rm d}m_d(\lambda)\right)^{1/2} \left(\int_{|\lambda|\leq R/2} |\widehat{h}(Rm-\lambda)|^2|Rm-\lambda|^2{\rm d}m_d(\lambda)\right)^{1/2} \nonumber\\ & \lesssim R^{d-2} \left(\int_{|\lambda|>R/2} |\widehat{h}(\lambda)|^2|\lambda|^2 {\rm d} m_d(\lambda) \right)^{1/2} = o(R^{d-2}) \nonumber
		\end{align}
		as $R\to \infty$. Note that in the third inequality we used the fact that $|m|\ge 1$. A similar bound can be obtained for ${\sf E}_2$ as
		\begin{align}
		\label{eq:variance_of_sobolev_bound_E_2}
		{\sf E}_2 &= R^{d-1}|m| \int_{|\lambda|\leq R/2} |\widehat{h}(\lambda)\widehat{h}(Rm-\lambda)| |\lambda| {\rm d} m_{d}(\lambda) \\ & \leq R^{d-1}|m| \left(\int_{|\lambda|\leq R/2}|\widehat{h}(\lambda)|^2|\lambda|^2{\rm d}m_d(\lambda)\right)^{1/2}\left(\int_{|\lambda|\leq R/2} |\widehat{h}(Rm-\lambda)|^2{\rm d}m_d(\lambda)\right)^{1/2} \nonumber \\ & \lesssim R^{d-1}|m| \left(\int_{|\lambda|\leq R/2} |\widehat{h}(Rm-\lambda)|^2{\rm d}m_d(\lambda) \right)^{1/2} \nonumber \\ & \lesssim R^{d-2}|m| \left(\int_{|\lambda|>R/2} |\widehat{h}(\lambda)|^2|\lambda|^2{\rm d}m_d(\lambda) \right)^{1/2} = o(R^{d-2}). \nonumber
		\end{align}
		Plugging the bounds (\ref{eq:variance_of_sobolev_bound_E_1}) and (\ref{eq:variance_of_sobolev_bound_E_2}) into (\ref{eq:first_bound_for_first_integral_sobolev_variance}) we can bound the first integral in (\ref{eq:variance_in_sobolev_split_remainder_integral}) as
		\begin{align}
		\label{eq:bound_on_first_integral_claim_sobolev_variance}
		\Big|R^d\int_{\RomanNumeralCaps{1}^\prime} \widehat{h}(\lambda)\widehat{h}(Rm-\lambda)&\left(e^{-a\pi^2|m|^2} - e^{-a\pi^2|\lambda|^2/R^2}e^{-a\pi^2|Rm-\lambda|^2/R^2}\right){\rm d} m_{d}(\lambda)\Big| \\ &\leq C_a e^{-a\pi^2|m|^2}|m| o(R^{d-2}). \nonumber
		\end{align}
		By the change of variables $\mu = Rm - \lambda$ we get that 
		\begin{align}
		\label{eq:bound_on_second_integral_claim_sobolev_variance}
		\Big|R^d\int_{\RomanNumeralCaps{2}^\prime} \widehat{h}(\lambda)\widehat{h}(Rm-\lambda)&\left(e^{-a\pi^2|m|^2} - e^{-a\pi^2|\lambda|^2/R^2}e^{-a\pi^2|Rm-\lambda|^2/R^2}\right){\rm d} m_{d}(\lambda)\Big| \\ &\leq C_a e^{-a\pi^2|m|^2}|m| o(R^{d-2}). \nonumber
		\end{align}
		also holds, so it remains to bound the third integral in (\ref{eq:variance_in_sobolev_split_remainder_integral}). Recall the definition of $\RomanNumeralCaps{3}^\prime$. We use Cauchy-Schwarz once more,
		\begin{align*}
		\Big|&R^d\int_{\RomanNumeralCaps{3}^\prime} \widehat{h}(\lambda)\widehat{h}(Rm-\lambda)\left(e^{-a\pi^2|m|^2} - e^{-a\pi^2|\lambda|^2/R^2}e^{-a\pi^2|Rm-\lambda|^2/R^2}\right){\rm d} m_{d}(\lambda)\Big| \\ & \lesssim e^{-a\pi^2|m|^2/2}R^{d} \left(\int_{\RomanNumeralCaps{3}^\prime}|\widehat{h}(\lambda)|^2{\rm d}m_d(\lambda)\right)^{1/2}\left(\int_{\RomanNumeralCaps{3}^\prime}|\widehat{h}(Rm-\lambda)|^2{\rm d}m_d(\lambda)\right)^{1/2} \\ & \lesssim e^{-a\pi^2|m|^2/2}R^{d} \int_{|\lambda|>R/2} |\widehat{h}(\lambda)|^2{\rm d}m_d(\lambda) \\ &\leq e^{-a\pi^2|m|^2/2}R^{d-2} \int_{|\lambda|>R/2} |\widehat{h}(\lambda)|^2|\lambda|^2{\rm d}m_d(\lambda) = e^{-a\pi^2|m|^2/2} o(R^{d-2}).
		\end{align*}
		Plugging (\ref{eq:bound_on_first_integral_claim_sobolev_variance}), (\ref{eq:bound_on_second_integral_claim_sobolev_variance}) and the above inequliaty into (\ref{eq:variance_in_sobolev_split_remainder_integral}) yields that
		\[
		|A_m(h,R)| \lesssim e^{-a\pi^2|m|^2/2}|m| o(R^{d-2})
		\]
		for all $m\in \mathbb{Z}^d \setminus \{0\}$.
	\end{proof}
	\begin{proof}[Proof of Theorem \ref{thm:variance_for_sobolev_space}]
		Relation (\ref{eq:variance_as_a_sum_after_scaling}) together with Claim \ref{claim:bound_on_remainder_variance_sobolev} yields that
		\begin{equation}
		\label{eq:variance_sobolev_bound_difference_from_main_term}
		\left|\var\left(N(h,R)\right) - A_0(h,R) \right| = o(R^{d-2}).
		\end{equation}
		Hence, to conclude the proof, we find the leading asymptotic term of $A_0(h,R)$. For every fixed $\lambda\in \mathbb{R}^d$ we have that
		\begin{equation*}
		\lim_{R\to\infty} R^2 \left(1-e^{-2a\pi^2|\lambda|^2/R^2}\right) = 2a\pi^2 |\lambda|^2.
		\end{equation*}
		Since $h\in H^1\left(\mathbb{R}^d\right)$ we may apply the dominated convergence theorem and see that
		\begin{align*}
		\lim_{R\to\infty} R^2 \int_{\mathbb{R}^d}|\widehat{h}(\lambda)|^2\left(1-e^{-2a\pi^2|\lambda|^2/R^2}\right) {\rm d} m_d(\lambda) &= 2a\pi^2\int_{\mathbb{R}^d} |\widehat{h}(\lambda)|^2 |\lambda|^2 {\rm d} m_d(\lambda) \\ & \stackrel{(\ref{eq:norm_of_gradiant_sobolev_space})}{=} \frac{a}{2}\int_{\mathbb{R}^d} |\nabla h(x)|^2 {\rm d} m_d(x).
		\end{align*}
		Combining the above with (\ref{eq:variance_sobolev_bound_difference_from_main_term}) yields that
		\begin{align*}
		\lim_{R\to \infty} \frac{\var\left(N(h,R)\right)}{R^{d-2}} = \lim_{R\to\infty} \frac{A_0(h,R)}{R^{d-2}} = \frac{a}{2}\int_{\mathbb{R}^d} |\nabla h(x)|^2 {\rm d} m_d(x).
		\end{align*}
	\end{proof}
	As mentioned in the introduction, the statement in Theorem \ref{thm:variance_for_sobolev_space} is false for the stationary process $W_{\sf s}$. In the next section we provide an example of a function $g\in \left(L^1\cap H^1\right)(\bR^d)$ such that $\var(N_{\sf s}(g,R))$ is large. Still, Theorem $\ref{thm:variance_linear_statstic_stat}$ implies that as long as $$|\widehat{h}(\lambda)| \lesssim (1+|\lambda|)^{-d/2-1}$$ we have that,
	\[
	\lim_{R\to \infty} \frac{\var\left(N_{\sf s}(h,R)\right)}{R^{d-2}} = \frac{a}{2}\int_{\mathbb{R}^d} |\nabla h(x)|^2 {\rm d} m_d(x).
	\]
	\subsection{Large variance for a function in the Sobolev space}
	\label{sec:large_variance_for_function_in_Sobolev_space} 
	It will be more illuminating (and easier) to construct the desired function $g\in (L^1\cap H^1)(\bR^d)$ on the Fourier side. Let $\rho\ge 0$ be a $C^\infty(\bR^d)$ bump function that is supported strictly inside $B_{1/2}$ and has $\rho(0) = 1$. Fix some $\eps>0$ and set
	\begin{equation}
		\label{eq:def_of_function_G}
		G(\lambda) = \sum_{m\in \mathcal{Z}\setminus\{0\}} c_m\rho\left(\frac{\lambda - m}{b_m}\right),
	\end{equation}
	where
	\[
	c_m = b_m |m|^{-1-\eps},  \qquad b_m = |m|^{-1/(d+1)},
	\]
	and $\mathcal{Z}$ is a $1$-dimensional sub-lattice of $\bZ^d$ given as in (\ref{eq:definition_of_the_sublattice}). Notice that $G(m) = c_m\rho(0) = c_m$ for all lattice points $m\in \mathcal{Z}\setminus\{0\}$. By Tonelli theorem we have,
	\begin{align*}
		\int_{\bR^d} G(\lambda) \dd m_d(\lambda) &= \sum_{m\in \mathcal{Z} \setminus\{0\}} c_m \int_{\bR^d} \rho\left(\frac{\lambda}{b_m}\right) \dd m_d(\lambda) \lesssim \sum_{m\in \mathcal{Z} \setminus\{0\}} c_m b_m^d, \\ 
		\int_{\bR^d} G(\lambda)^2 \dd m_d(\lambda) &= \sum_{m\in \mathcal{Z} \setminus\{0\}} c_m^2 \int_{\bR^d} \rho^2\left(\frac{\lambda}{b_m}\right) \dd m_d(\lambda) \lesssim \sum_{m\in \mathcal{Z} \setminus\{0\}} c_m^2 b_m^d,
	\end{align*}
	which implies that $G\in (L^1 \cap L^2)(\bR^d)$. We want to show further that $G(\lambda)|\lambda|\in L^2(\bR^d)$. Indeed,
	\begin{align*}
		\int_{\bR^d} G(\lambda)^2|\lambda|^2 \dd m_d(\lambda) &= \sum_{m\in \mathcal{Z} \setminus\{0\}} c_m^2 \int_{\bR^d} \rho^2\left(\frac{\lambda-m}{b_m}\right)|\lambda|^2 \dd m_d(\lambda) \\ &= \sum_{m\in \mathcal{Z}\setminus\{0\}} c_m^2b_m^d \int_{\bR^d} \rho^2\left(\mu\right)|b_m\mu + m|^2 \dd m_d(\mu) \\ &\lesssim \sum_{m\in \mathcal{Z} \setminus\{0\}} b_m^{d+2} |m|^{-2\eps} \lesssim \sum_{\ell=1}^{\infty} \frac{1}{\ell^{1+2\e}}<\infty .
	\end{align*}
	For $\alpha=(\alpha_1,\ldots,\alpha_n)\in \bN^n$ we use the standard notation for partial derivatives: 
	\[
	|\alpha| \eqdef \sum_{j=1}^n{\alpha_n},\qquad \frac{\partial^\alpha}{(\partial \lambda)^\alpha} \eqdef \left(\frac{\partial}{\partial \lambda_1}\right)^{\alpha_1} \cdots \left(\frac{\partial}{\partial \lambda_n}\right)^{\alpha_n}.
	\] 
	Let $g$ be the inverse Fourier transform of $G$, given by
	\[
	g(x) \eqdef \int_{\bR^d} G(\lambda) e(\inner{\lambda}{x}) \dd m_d(x).
	\]
	First, we verify that $g\in L^1(\bR^d)$. Indeed, the sum (\ref{eq:def_of_function_G}) defining $G$ is absolutely and uniformly convergent, we may differentiate term-wise and see that
	\[
	\frac{\partial^\alpha}{(\partial \lambda)^\alpha}G(\lambda) = \sum_{m\in \mathcal{Z}\setminus\{0\}} \frac{c_m}{b_m^{|\alpha|}}\rho^{(\alpha)}\left(\frac{\lambda - m}{b_m}\right).
	\] 
	Therefore,
	\begin{equation}
		\label{eq:upperbound_for_l1_norm_of_derivatives_sobolev_counter_example}
		\int_{\bR^d} \left|\frac{\partial^\alpha}{(\partial \lambda)^\alpha}G(\lambda)\right| {\rm d}m_d(\lambda)  \lesssim \sum_{m\in \mathcal{Z}\setminus\{0\}} \frac{c_m}{b_m^{|\alpha|-d}} < \infty
	\end{equation}
	provided that $|\alpha| \leq d+1$ (recall (\ref{eq:def_of_function_G})). By a simple integration by parts argument (see for example \cite[Theorem~8.22]{folland}) we conclude from (\ref{eq:upperbound_for_l1_norm_of_derivatives_sobolev_counter_example}) that $$|g(x)|\lesssim (1 + |x|)^{-(d+1)}$$ which in turn implies that $g\in L^1(\bR^d)$. Furthermore, since $\widehat{g}(\lambda) = G(\lambda)$, we have that $g\in H^1(\bR^d)$. We now examine $\var(N_{\sf s}(g,R))$. Clearly,
	\begin{align*}
		R^d\int_{\bR^d} &|\widehat{g}(\lambda)|^2\left(1-e^{-2a\pi^2|\lambda|^2/R^2}\right) \dd m_d(\lambda) \\ &= R^d\int_{\bR^d} |G(\lambda)|^2\left(1-e^{-2a\pi^2|\lambda|^2/R^2}\right) \dd m_d(\lambda)  = \mathcal{O}(R^{d-2}).
	\end{align*}
	Suppose that $R\to\infty$ on the integers. Then, $Rm\in \mathcal{Z}$ for all $m\in \mathcal{Z}$ and by Theorem $\ref{thm:variance_linear_statstic_stat}$  we have the lower bound
	\begin{align*}
		\var\left(N_{\sf s}(g,R)\right) & \ge R^{2d} \sum_{m\in \bZ^d\setminus \{0\}} e^{-2a\pi^2|m|^2} |\widehat{g}(Rm)|^2 \\ &= R^{2d} \sum_{m\in \mathcal{Z} \setminus \{0\}} e^{-2a\pi^2|m|^2} |G(Rm)|^2  \\ &\ge 2R^{2d} \sum_{\ell = 1}^{\infty} e^{-2a\pi^2\ell^2} (R\ell)^{-2(1+\eps+\frac{1}{d+1})} \gtrsim R^{2(d-1) -2\eps - 2/(d+1)}. 
	\end{align*} 
	It remains to observe that
	$
	2(d-1) -2\eps - 2/(d+1) > d-2
	$
	for all $d\ge 2$ provided that $\eps\in (0,1/4)$.
	This observation immediately gives
	\[
	\limsup_{R\to \infty} \frac{\var\left(N_{\sf s}(g,R)\right)}{R^{d-2}} = +\infty.
	\]
	\section{Indicator functions of Convex sets with smooth boundary}
	\label{sec:convex_sets_with_smooth_boundary}
	In this section we give the proof of Theorem \ref{thm:variance_for_convex_sets}. Recall that $K$ is a compact convex set such that $\partial K$ is a smooth manifold with non-vanishing Gaussian curvature, and denote by $\sigma_{d-1}$ the induced Lebesgue surface measure on $\partial K$. With the above assumptions on $K$, we have that for all $\psi\in C^\infty(\mathbb{R}^d)$,
	\begin{equation}
		\label{eq:bound_on_fourier_transfom_of_manifold}
		\left|\int_{\partial K} \psi(x)e(-\inner{\lambda}{x}) {\rm d}\sigma_{d-1}(x) \right| \lesssim \left(1 + |\lambda|\right)^{-(d-1)/2}
	\end{equation}
	where the implicit constant depends only on $\psi$ and the Gaussian curvature of $\partial K$, see \cite[Theorem~7.7.14]{hormander} or \cite[Theorem~1.2.1]{sogge}. In fact, one can recover the upper bound (\ref{eq:bound_on_fourier_transform_of_smooth_set}) on the Fourier transform of $\mathbf{1}_{K}$ using (\ref{eq:bound_on_fourier_transfom_of_manifold}).
	
	Similarly to the proof of Theorem \ref{thm:variance_for_sobolev_space}, the strategy for the proof of Theorem \ref{thm:variance_for_convex_sets} is to use the formula (\ref{eq:variance_as_a_sum_after_scaling}) and show that the term $m=0$ dominates the rest of the sum as $R\to\infty$. First, we find the leading order asymptotic for the term $m=0$.
	\begin{claim}
		\label{claim:leading_order_for_var_convex_sets}
		Suppose that $K\subset \mathbb{R}^d$ is a compact convex set. Then,
		\begin{equation*}
			\lim_{R\to\infty} \frac{A_{0}(\mathbf{1}_K,R)}{R^{d-1}} = \sqrt{\frac{a}{2\pi}} \cdot \sigma_{d-1}(\partial K).
		\end{equation*}
	\end{claim}
	Note that we do not assume anything on $\partial K$ in the statement of Claim \ref{claim:leading_order_for_var_convex_sets}. Indeed, we may apply the claim in Section \ref{sec:variance_of_number_of_points_inside_large_cube} also in the case $K=[-\frac{1}{2},\frac{1}{2}]^d$. 
	\begin{proof}
		This claim follows from a simple computation. By Plancherel's formula (see \cite[Theorem~8.29]{folland}),
		\begin{align*}
		A_{0}(\mathbf{1}_K,R) &= R^{d} \int_{\mathbb{R}^d} |\widehat{\mathbf{1}_K}(\lambda)|^2(1-e^{-2a\pi^2|\lambda|^2/R^2}) {\rm d} m_{d}(\lambda) \\ &= R^d \left[m_d(K) - \int_{\mathbb{R}^d} |\widehat{\mathbf{1}_K}(\lambda)|^2e^{-2a\pi^2|\lambda|^2/R^2} {\rm d} m_{d}(\lambda)\right] \\ &= R^d \left[m_d(K) - \frac{R^d}{(2a\pi)^{d/2}} \int_{\mathbb{R}^d} \left(\mathbf{1}_K\ast \mathbf{1}_K\right)(x) e^{-R^2|x|^2/2a} {\rm d} m_{d}(x)\right] \\ &= \frac{R^d}{(2a\pi)^{d/2}} \int_{\mathbb{R}^d} \left[m_d(K) - \left(\mathbf{1}_K\ast \mathbf{1}_K\right)\left(\frac{x}{R}\right)\right] e^{-|x|^2/2a} {\rm d} m_{d}(x).
		\end{align*}
		Since $K$ is convex, we can compute the directional derivative at the origin of the function $\left(\mathbf{1}_K\ast \mathbf{1}_K\right)(\cdot)$, see \cite[Proposition~4.3.1]{matheron}. By Taylor expansion, for every fixed $x\in \mathbb{R}^d$,
		\[
		m_d(K) - \left(\mathbf{1}_K\ast \mathbf{1}_K\right)\left(\frac{x}{R}\right) = m_{d-1}\left(P_{x}(K)\right) \frac{|x|}{R}  + \mathcal{O}\left(\frac{|x|^2}{R^2}\right),
		\]
		where $P_x$ is the linear projection onto the hyperplane $x^\perp = \left\{v\in \mathbb{R}^d \mid \inner{x}{v}=0\right\}$. By the dominated convergence theorem (notice that $\mathbf{1}_K\ast \mathbf{1}_K$ has compact support) we see that
		\begin{align*}
		A_{0}(h,R) &= R^{d-1} \int_{\mathbb{R}^d} |x| m_{d-1}\left(P_{x}(K)\right) e^{-|x|^2/2a} \frac{{\rm d} m_{d}(x)}{(2a\pi)^{d/2}} + \mathcal{O}(R^{d-2}) \\ &= R^{d-1}\left(\int_{0}^{\infty} t^{d} e^{-t^2/2a} \frac{{\rm d}t}{(2a\pi)^{d/2}} \right) \left(\int_{\mathbb{S}^{d-1}} m_{d-1}\left(P_{u}(K)\right) {\rm d} \sigma_{d-1}(u) \right) + \mathcal{O}(R^{d-2}) \\ &=  R^{d-1} \cdot \sqrt{\frac{a}{2}}\frac{\Gamma\left(\frac{d+1}{2}\right)}{\pi^{d/2}} \left(\int_{\mathbb{S}^{d-1}} m_{d-1}\left(P_{u}(K)\right) {\rm d} \sigma_{d-1}(u) \right) + \mathcal{O}(R^{d-2}),
		\end{align*}
		where $\mathbb{S}^{d-1} = \{u\in \mathbb{R}^d : |u|=1\}$ and $\sigma_{d-1}$ is the induced surface measure on it. Finally, by Cauchy's surface area formula \cite[eq. (5.73), p.~301]{schneider}
		\[
		\int_{\mathbb{S}^{d-1}} m_{d-1}\left(P_{u}(K)\right) {\rm d} \sigma_{d-1}(u) = m_{d-1}(B) \sigma_{d-1}(\partial K) = \frac{\pi^{(d-1)/2}}{\Gamma\left(\frac{d+1}{2}\right)} \sigma_{d-1}(\partial K)
		\]
		which finishes the proof of the claim.
	\end{proof}
	\begin{remark*}
		We relate Claim \ref{claim:leading_order_for_var_convex_sets} to the discussion from the introduction of this paper. Suppose we consider i.i.d. perturbations of the lattice points, all with common distribution $\xi$ (which, for the moment, is not necessarily a symmetric Gaussian). Then, provided that $\xi$ has a density, one can prove along the lines of the proof of Claim \ref{claim:leading_order_for_var_convex_sets} that $$\lim_{R\rightarrow\infty} \frac{A_0(\mathbf{1}_K,R)}{R^{d-1}} =  \E\left[|\alpha|P_\alpha(K)\right],$$ 
		where $\alpha = \xi^\prime - \xi^{\prime\prime}$ and $\xi^\prime,\xi^{\prime\prime}$ are independent copies of $\xi$. The same limiting constant appeared in the paper by G\'{a}cs and Sz\'{a}sz \cite{gacs_szasz}, where an extra averaging of the variance was considered (i.e. integrated over all possible translations of $K_R$).
	\end{remark*}
	\begin{claim}
		\label{claim:stationary_phase_upper_bound_convolution_with_convex}
		Suppose that $K\subset \mathbb{R}^d$ is a compact convex set such that $\partial K$ is a smooth closed manifold with nowhere vanishing Gaussian curvature, and that $m\in \bZ^d\setminus 
		\{0\}$. Then
		\begin{equation*}
		\left|\int_{\mathbb{R}^d} \widehat{\mathbf{1}_K}(\lambda)\widehat{\mathbf{1}_K}(Rm-\lambda)e^{-a\pi^2(|\lambda|^2/R^2 - |Rm-\lambda|^2/R^2)}{\rm d} m_{d}(\lambda)\right| \leq C_a e^{-a\pi^2|m|^2/2} (1+R|m|)^{-(d+1)/2}.
		\end{equation*}
	\end{claim}
	We postpone the proof of Claim \ref{claim:stationary_phase_upper_bound_convolution_with_convex} and first prove the theorem.
	\begin{proof}[Proof of Theorem \ref{thm:variance_for_convex_sets}]
		Fix $m\in \mathbb{Z}^d\setminus \{0\}$ for the moment. Notice that
		\[
		A_m(\mathbf{1}_K,R) = A_{m}^{1}(\mathbf{1}_K,R)+ A_{m}^{2}(\mathbf{1}_K,R)
		\]
		where,
		\begin{align*}
			A_m^{1}(\mathbf{1}_K,R) &= e^{-a\pi^2|m|^2} R^d\int_{\mathbb{R}^d} \widehat{\mathbf{1}_K}(\lambda)\widehat{\mathbf{1}_K}(Rm-\lambda) {\rm d} m_{d}(\lambda) \\ A_m^2(\mathbf{1}_K,R) &= R^d \int_{\mathbb{R}^d} \widehat{\mathbf{1}_K}(\lambda)\widehat{\mathbf{1}_K}(Rm-\lambda)e^{-a\pi^2(|\lambda|^2/R^2 - |Rm-\lambda|^2/R^2)}{\rm d} m_{d}(\lambda).
		\end{align*}
		By Plancherel's formula,
		\begin{equation*}
			\int_{\mathbb{R}^d} \widehat{\mathbf{1}_K}(\lambda)\widehat{\mathbf{1}_K}(Rm-\lambda) {\rm d} m_{d}(\lambda) = \int_{\mathbb{R}^d} \mathbf{1}_K(x) \overline{e(\inner{x}{Rm})} {\rm d} m_d(x) = \widehat{\mathbf{1}_K}(Rm).
		\end{equation*}
		Whence, by the upper bound on the Fourier transform of $\mathbf{1}_{K}$ (\ref{eq:bound_on_fourier_transform_of_smooth_set}), we have that
		\[
		|A_m^{1}(\mathbf{1}_K,R)| \leq C_a e^{-a\pi^2|m|^2}|m|^{-(d+1)/2} R^{(d-1)/2}.
		\]
		The above inequality, combined with Claim \ref{claim:stationary_phase_upper_bound_convolution_with_convex} gives that
		\[
		|A_m(\mathbf{1}_K,R)| \leq |A_{m}^{1}(\mathbf{1}_K,R)| + |A_{m}^{2}(\mathbf{1}_K,R)| \leq C_a e^{-a\pi^2|m|^2/2} R^{(d-1)/2}.
		\]
		By (\ref{eq:variance_as_a_sum_after_scaling}), we immediately obtain that
		\begin{equation*}
			\label{eq:inequality_between_variance_and_main_term_convex}
			\left|\var\left(\mathbf{n}(K_R)\right) - A_0(\mathbf{1}_K,R)\right| \leq C_a R^{(d-1)/2},
		\end{equation*} 
		which, together with Claim \ref{claim:leading_order_for_var_convex_sets}, finishes the proof.
	\end{proof}
	\begin{proof}[Proof of Claim \ref{claim:stationary_phase_upper_bound_convolution_with_convex}]
		By the parallelogram law
		\[
		\left|\lambda\right|^2 +\left|Rm-\lambda \right|^2 = \frac{R^2|m|^2 + \left|Rm -2\lambda\right|^2}{2}
		\]
		so it will be enough to prove that
		\begin{equation}
			\label{eq:inequallity_that_finish_claim_convex_sets}
			\left|\int_{\mathbb{R}^d} \widehat{\mathbf{1}_K}(\lambda)\widehat{\mathbf{1}_K}(Rm-\lambda)e^{-2a\pi^2(|\lambda-Rm/2|^2/R^2}{\rm d} m_{d}(\lambda)\right| \lesssim (1+R|m|)^{-(d+1)/2}.
		\end{equation}
		This we do in what follows. We use Plancherel's formula and change of variables to get that
		\begin{align*}
			&\left|\int_{\mathbb{R}^d} \widehat{\mathbf{1}_K}(\lambda)\widehat{\mathbf{1}_K}(Rm-\lambda)e^{-2a\pi^2(|\lambda-Rm/2|^2/R^2}{\rm d} m_{d}(\lambda)\right| \\ &= \frac{R^d}{(2\pi a)^{d/2}}\left|\int_{\mathbb{R}^d} e^{-R^2|x|^2/2a}e(R\inner{m}{x}/2)\left(\int_{\mathbb{R}^d} \mathbf{1}_{K}(y)\mathbf{1}_{K}(x-y)\overline{e(R\inner{m}{y})} {\rm d}m_d(y) \right) {\rm d}m_d(x) \right| \\ & \lesssim \int_{\mathbb{R}^d} e^{-|x|^2} \left|\int_{\mathbb{R}^d} \mathbf{1}_{K}(y)\mathbf{1}_{K}(x-y)e(-R\inner{m}{y}) {\rm d}m_d(y)\right| {\rm d}m_d(x) \lesssim \sup_{x\in\mathbb{R}^d }\left|\widehat{\mathbf{1}_{\Lambda_x}} (Rm)\right|,
		\end{align*}
		where $\Lambda_x \eqdef K \cap (x-K)$. The boundary of $\Lambda_x$ consists of two parts:
		\[
		\partial \Lambda_x = \partial \Lambda_x^{\prime} \cap \partial \Lambda_x^{\prime\prime}
		\]
		where $\partial \Lambda_x^{\prime} = \partial K \cap (x-K)$ and $\partial \Lambda_x^{\prime\prime} = K \cap \partial(x-K)$. Denote by $n(y)$ the outward normal to the surface $\partial \Lambda_x$ at the point $y$. By applying the divergence theorem with the vector field
		\[
		y\mapsto \frac{e\left(-\inner{Rm}{y}\right)}{-2\pi i R |m|} \cdot \frac{m}{|m|},\qquad y\in \mathbb{R}^d,
		\]
		we obtain that
		\begin{align}
			\label{eq:fourier_transform_of_set_with_divergence}
			\widehat{\mathbf{1}_{\Lambda_x}}(Rm) &= \int_{\Lambda_x} e(-\inner{Rm}{y}) {\rm d}m_d(y) \nonumber \\ &= \frac{1}{-2\pi i R |m|}\int_{\partial \Lambda_x} e\left(-\inner{Rm}{y}\right) \inner{\frac{m}{|m|}}{n(y)} {\rm d} \sigma_{d-1}(y) \nonumber \\ &= \frac{1}{-2\pi i R |m|}\left(\int_{\partial \Lambda_x^{\prime}} + \int_{\partial\Lambda_x^{\prime\prime}}\right) (\cdots) {\rm d}\sigma_{d-1}(y) \eqdef \frac{1}{-2\pi i R |m|}\left( {\sf J}^\prime + {\sf J}^{\prime\prime} \right).
		\end{align}
		Notice that $\partial \Lambda_x^{\prime}$ is a smooth $(d-1)$-manifold which inherits the Gaussian curvature of the manifold $\partial K$. Hence, we can apply inequality (\ref{eq:bound_on_fourier_transfom_of_manifold}) with $\psi(y)  = \inner{\frac{m}{|m|}}{n(y)}$ and obtain that
		\[
		|{\sf J}^\prime| = \left|\int_{\partial \Lambda_x^{\prime}} e\left(-\inner{Rm}{y}\right) \Big\langle\frac{m}{|m|},n(y) \Big\rangle {\rm d} \sigma_{d-1}(y) \right| \lesssim \left(1+R|m|\right)^{-(d-1)/2},
		\]
		and the constant does not depend on $x$. Similarly we have that
		$
		|{\sf J}^{\prime\prime}| \lesssim \left(1+R|m|\right)^{-(d-1)/2}
		$
		and hence, by plugging into (\ref{eq:fourier_transform_of_set_with_divergence}) we obtain that
		\[
		\sup_{x\in\mathbb{R}^d } \left|\widehat{\mathbf{1}_{\Lambda_x}}(Rm) \right| \lesssim \left(1+R|m|\right)^{-(d+1)/2}
		\]
		This proves inequality (\ref{eq:inequallity_that_finish_claim_convex_sets}) and hence the claim.
	\end{proof}
	
	\subsection*{Acknowledgments}
	I am deeply grateful to my advisors, Alon Nishry and Mikhail Sodin, for their guidance throughout this work and for many stimulating conversations. I also thank Ofir Karin and Aron Wennman for helpful discussions.

\end{document}